\documentclass[a4paper,reqno,12pt]{amsart}
\usepackage{a4wide}
\usepackage[utf8]{inputenc}
\usepackage[english]{babel}
\usepackage{amsmath,amssymb,amsthm,graphicx,mathrsfs,url}

\usepackage{enumitem}
\usepackage{mathtools}
\usepackage[usenames,dvipsnames]{xcolor}
\usepackage[colorlinks=true,linkcolor=blue,citecolor=Green]{hyperref}
\usepackage{tikz}
\tikzstyle{help lines}+=[dotted]
\usetikzlibrary{automata,positioning,angles,quotes,babel,calc,arrows,shapes,intersections,decorations.pathreplacing,backgrounds,through,backgrounds}

\newcommand{\C}{\mathbb{C}}
\newcommand{\R}{\mathbb{R}}

\newcommand{\D}{\mathcal{D}}
\newcommand{\N}{\mathbb{N}}

\newcommand{\norm}[1]{\left\Vert#1\right\Vert}

\let\Im=\Imag

\let\Re=\Real

\DeclareMathOperator{\supp}{supp}

\def\Ddots{\mathinner{\mkern1mu\raise\p@
\vbox{\kern7\p@\hbox{.}}\mkern2mu
\raise4\p@\hbox{.}\mkern2mu\raise7\p@\hbox{.}\mkern1mu}}
\makeatother

\newcommand{\eklm}[1]{\left\langle #1 \right\rangle}

\renewcommand{\d}{\mathrm{d}}

\renewcommand{\epsilon}{\vararepsilon}

\newcommand{\bdm}{\begin{displaymath}}
\newcommand{\edm}{\end{displaymath}}
\newcommand{\bq}{\begin{equation}}
\newcommand{\eq}{\end{equation}}
\newcommand{\bqn}{\begin{equation*}}
\newcommand{\eqn}{\end{equation*}}

\newcommand{\Cinft}{{C^{\infty}}}
\newcommand{\CT}{{C^{\infty}_c}}

\renewcommand{\L}{{\rm L}}

\newcommand{\g}{{\bf \mathfrak g}}

\newcommand{\aL}{{\bf \mathfrak a}}
\newcommand{\nL}{{\bf \mathfrak n}}
\renewcommand{\k}{{\bf \mathfrak k}}

\newcommand{\p}{{\bf \mathfrak p}}

\renewcommand{\Im}{\mathrm{Im}\,}
\renewcommand{\Re}{\mathrm{Re}\,}

\newcommand{\gam}{\Gamma\backslash}

\numberwithin{equation}{section}

\title[Quantum resonances and scattering poles]{Quantum resonances and scattering poles\\ of classical rank one locally symmetric spaces}

\theoremstyle{plain}
  \newtheorem{thm}{Theorem}[section]
  \newtheorem{lem}[thm]{Lemma}
  \newtheorem{prop}[thm]{Proposition}

  \theoremstyle{definition}
  \newtheorem{defn}[thm]{Definition}

  \theoremstyle{remark}
  \newtheorem{rem}[thm]{Remark}

\author[B.~Delarue]{Benjamin Delarue}
\email{bdelarue@math.upb.de}
 \author[J.~Hilgert]{Joachim Hilgert}
\email{hilgert@math.upb.de}
 \address{Universit\"at Paderborn, Warburger Str. 100, 33098 Paderborn, Germany}
\date{\today}

\begin{document}

\begin{abstract}
For negatively curved symmetric spaces it is known from \cite{HHP19} that the poles of the scattering matrices defined via the standard intertwining operators for the spherical principal representations of the isometry group are either given as poles of the intertwining operators or as quantum resonances, i.e. poles of the meromorphically continued resolvents of the Laplace-Beltrami operator. We extend this result to classical locally symmetric spaces of negative curvature with convex-cocompact fundamental group using results of Bunke and Olbrich. The method of proof forces us to exclude the spectral parameters corresponding to singular Poisson transforms. 
\end{abstract}

\maketitle

\section{Introduction}

Let $G$ be a non-compact connected classical real simple Lie group of real rank $1$.  We assume for simplicity that the center of $G$ is finite and fix a maximal compact subgroup $K$ of $G$. The quotient $X:=G/K$ is a non-compact Riemannian symmetric space of rank $1$. 

Suppose that $\Gamma\subset G$ is a discrete torsion-free subgroup. We denote the limit set of $\Gamma$ in the geodesic boundary $\partial X$ of $X$ by $\Lambda$ and set $\Omega:=\partial X\setminus \Lambda$. Further, we assume that $\Gamma$ is a convex-cocompact subgroup, by which we mean that the quotient $\gam(X\cup \Omega)$ is a compact manifold with boundary. $\Omega$ is then called the \emph{domain of discontinuity} of $\Gamma$ in $\partial X$ and $\Gamma\backslash \Omega$ is  called the \emph{boundary at infinity} of $\Gamma\backslash X$. We will always assume that $\Gamma$ is not cocompact, which means that the boundary of $\gam(X\cup \Omega)$ is non-empty.

Let $\Delta_X:\CT(X)\to \CT(X)$ denote the positive Laplacian of $X$ which is a $G$-invariant differential operator. Furthermore, there is the positive Laplacian $\Delta_ {\gam X}:\CT(\gam X)\to \CT(\gam X)$ of $\gam X$. Since $X$ and $\gam X$ are complete Riemannian manifolds, $\Delta_X$ and $\Delta_{\gam X}$ are essentially self-adjoint operators in $\L^2(X)$ and $\L^2(\gam X)$, respectively; we denote their unique self-adjoint extensions again by $\Delta_X$ and $\Delta_{\gam X}$.

The \emph{quantum resonances} are the poles of the meromorphically continued modified $\L^2$-resolvents
\begin{align}
\begin{split}
R_\zeta&:=(\Delta_X-\norm{\rho}^2-\zeta^2)^{-1}:\L^2(X)\to \L^2(X),\\
 R^\Gamma_\zeta&:=(\Delta_{\gam X}-\norm{\rho}^2-\zeta^2)^{-1}:\L^2(\gam X)\to \L^2(\gam X)\label{eq:RL2},\end{split}
\end{align}
which are holomorphic operator families on the \emph{physical half-plane} $\mathcal{H}_\mathrm{phys}:=\{\zeta\in \C\,|\, \Im \zeta >0 \}$ and on a possibly slightly smaller half-plane $\mathcal{H}_\mathrm{phys}^\Gamma\subset \mathcal{H}_\mathrm{phys}$, respectively (see Section~\ref{sec:resolventkernels} for details). Here $\rho$ denotes the usual half-sum of positive restricted roots and $\norm{\cdot}$ is the norm defined by the Killing form, see Section~\ref{sec:ressccc} for the setup of notation.

The \emph{scattering poles} mentioned in the title are the poles of the meromorphic families $S_\lambda$ and $S_\lambda^\Gamma$ of scattering matrices which in our situation can be constructed from the standard intertwining operators for the spherical principal series representations of $G$ using meromorphic families of extension operators constructed by Bunke and Olbrich in \cite{BO00}. For technical reasons they had to exclude the case $G=F_4^{-20}$. This is the only reason why also we restrict our attention to the classical rank one spaces.

The relation between quantum resonances and scattering poles for the global rank one symmetric spaces was described explicitly in \cite{HHP19}. The quantum resonances are\footnote{Note that the physical halfplane in \cite{HHP19} is given by $\Im \zeta<0$, whereas we take $\Im \zeta>0$, which means that we need to flip the sign of $\zeta$ when formulating certain statements from \cite{HHP19} in our setting.} the zeros of the function $c(i\zeta)c(-i\zeta)$ with 
$\Im \zeta<0$, where $c$ is a variant of the Harish-Chandra $c$-function. The zeros of $c$ are known (see e.g.\ \cite[Rem.~5.1]{HHP19}). The main result \cite[Thm.~7.1]{HHP19} of that paper says that the scattering poles outside of $\mathcal{H}_\mathrm{phys}$ are precisely the quantum resonances (with equal and finite multiplicities), whereas the scattering poles which are not quantum resonances are the poles of the family of standard intertwining operators. 

Earlier results are available (see \cite{GZ97,BP02,Gu05}) for the case of asymptotically real hyperbolic manifolds. While they take the scattering matrices from geometric scattering theory and use microlocal analysis of manifolds with corners, the basic means of comparison between resonances and scattering poles is, just as in the harmonic analysis approach of \cite{HHP19}, the possibility to take boundary values of resolvent kernels yielding modifications of the scattering kernels. Borthwick's book \cite{Bo16} contains a very readable account of these results for the case of asymptotically hyperbolic surfaces.

That the results of Bunke and Olbrich could be used to extend the relation between quantum resonances and scattering poles to convex-cocompact locally symmetric spaces of rank one was suggested in \cite[Problem~6.5]{Hi23} based on the fact that in \cite{BO12} Bunke and Olbrich studied the meromorphic extention of the resolvent kernels as well as the scattering poles for the classical locally symmetric spaces with negative curvature and convex-cocompact fundamental group.  

Our point of departure is \cite[Lemma 11]{BO12}, where they show that the difference between the resolvent kernels on $X$ and the lift of the resolvent kernels on  $\gam X$ has  boundary values on $\partial X\times \partial X$ which depend meromorphically on the spectral parameters. A careful analysis of the proof of this lemma yields a meromorphic function $a$ relating the boundary values of the lifted resolvent kernels of $\gam X$ and the lifted kernels of the scattering matrices away from the diagonal. Our technical Lemma~\ref{lem:poles-diagonal} (which may be of independent interest) shows that no extra singularities come from the diagonal and the results of \cite{HHP19} allow to determine the function $a$.  Thus we are lead to a generalization of the statement about the scattering poles being quantum resonances except for the exceptional spectral parameters for which the Poisson transform is singular so that no boundary value map is defined. More precisely, using the physical half-plane $\mathcal{H}^\Gamma_\mathrm{phys}$ depending on $\Gamma$,  our main result can be formulated as follows, where $\alpha_0$ denotes the unique reduced root in a chosen positive system (see Theorem~\ref{main theorem}).

\medskip
\noindent
{\bf Theorem.} {\it Outside the set $-\frac{i}{2}\norm{\alpha_0}\N_0$, the scattering poles which are not in the physical half-plane $\mathcal{H}^\Gamma_\mathrm{phys}$ are precisely the quantum resonances. Moreover, if $\zeta_0\not\in -\frac{i}{2}\norm{\alpha_0}\N_0$ is a simple pole of the quantum resolvent as well as the scattering matrix, then the boundary value map induces an isomorphism between the images of the two residue operators at $\zeta_0$.}
\medskip

In general, quantum resonances will not be simple poles of the resolvent, so our result on residues is incomplete. In the case of asymptotically hyperbolic manifolds complete results are available. While \cite{BP02} achieve this by adding potentials to split spectral lines, \cite{GZ97,Gu05} use Gohberg-Sigal normal forms (see \cite{GS71}). It seems plausible to us that these techniques could also be employed in our situation, but we have not made a serious effort to do so as it would have lengthened the paper in a disproportionate way. 

\subsection*{Acknowledgements} Funded by the Deutsche Forschungsgemeinschaft (DFG, German Research Foundation) – Project-ID 491392403 – TRR 358. BD has received further funding from the Deutsche Forschungsgemeinschaft through the Priority Program (SPP) 2026 ``Geometry at Infinity''.
\section{Setup}\label{sec:ressccc}

 Let $\g=\k\oplus \p$ be the Cartan decomposition of the Lie algebra of $G$, where $\k$ is the Lie algebra of $K$, and denote by $\eklm{\cdot,\cdot}$ the inner product on $\g$ defined by the Killing form and the Cartan involution $\theta$, and by $\norm{\cdot}$ the associated norm.  Let $\aL\subset \p$ be a maximal abelian subspace, $\Sigma\subset \aL^\ast$ the root system of $(\g,\aL)$, and fix a choice $\Sigma^+\subset \Sigma$ of positive roots, corresponding to a choice of positive Weyl chamber $\aL^+\subset \aL$. Denote by $\alpha_0\in \Sigma^+$ the unique reduced root. Since $\dim \aL=1$, there is precisely one element $H_0\in \aL^+$ with $\norm{H_0}=1$; it satisfies $\alpha_0(H_0)=\norm{\alpha_0}$.  
Let $\nL:=\bigoplus_{\alpha\in \Sigma^+}\g_\alpha$ be the nilpotent Lie algebra spanned by the root spaces $\g_\alpha$ and  $\rho:=\frac{1}{2}\sum_{\alpha\in \Sigma^+}m_\alpha \alpha\in \aL^\ast$ the half-sum of the positive roots, weighted according to their multiplicities $m_\alpha:=\dim \g_\alpha$.

Let $A, N\subset G$ be the analytic subgroups with Lie algebras $\aL$ and $\nL$, respectively, and let $M:=Z_K(A)$ be the centralizer of $A$ in $K$. Then the Iwasawa decomposition of $G$ is given by $G=KAN$ and  $P:=MAN$ is a minimal parabolic subgroup of $G$. We denote by $\exp:\g\to G$ the exponential map and by $\kappa:G\to K$, $H:G\to \aL$ the smooth maps such that $g\in\kappa(g)\exp(H(g))N$ for all $g\in G$. For $Y\in \aL$, $a=\exp(Y)\in A$, $\lambda\in \aL_\C^*:=\aL^\ast\otimes_\R\C$, we write $a^\lambda:=e^{\lambda(Y)}\in \C$. We fix the Haar measure $dk$ on $K$ of volume $1$.

\section{Resolvent kernels}\label{sec:resolventkernels}

It is known from \cite{HC76} and \cite{BO00} that the continuous spectra of $\Delta_{X}$ as well as $\Delta_{\gam X}$ are given by the interval $[\norm{\rho}^2,\infty)$, with $\Delta_{\gam X}$ having at most finitely many eigenvalues outside this interval. $\norm{\rho}^2$ is actually the minimum of the spectrum of  $\Delta_X$, but in 
order to describe the minimum of the spectrum of  $\Delta_{\gam X}$ we have to introduce the \emph{critical exponent}
\bq \begin{split}
\delta_\Gamma &:=\inf\Big\{s>0\,\Big|\, \sum_{\gamma \in \Gamma}e^{-sd_X(o,\gamma o)}<+\infty\Big\}\\\label{eq:criticalexponent}
&\phantom{:}=\inf\Big\{s>0\,\Big|\,\forall\; x,y\in X: \sum_{\gamma \in \Gamma}e^{-sd_X(x,\gamma y)}<+\infty\Big\},\end{split}
\eq
where  $d_X:X\times X\to \R$ is the Riemannian distance and $o:=K\in G/K=X$ is the canonical base point. The critical exponent lies in the interval $[0,2\norm{\rho}]$ (see \cite[Lemma~3.1]{We08}).

By \cite[Thm.~4.2]{Co90} (see also \cite[Thm.~1]{AZ22} as well as \cite{HC76} and \cite{BO00}) the minimum of the spectrum of  $\Delta_{\gam X}$ is $\norm{\rho}^2$ if $0\leq \delta_\Gamma\leq \norm{\rho}$, whereas in the case $\norm{\rho}< \delta_\Gamma\leq 2\norm{\rho}$ the latter minimum is given by $\norm{\rho}^2-(\delta_\Gamma-\norm{\rho})^2$. Thus, while the family $\zeta\mapsto  R_\zeta$ is holomorphic on the physical halfplane $\mathcal{H}_\mathrm{phys}$, the family $\zeta\mapsto   R_\zeta^\Gamma$ is holomorphic only on the halfplane 
\bq
\mathcal{H}^\Gamma_\mathrm{phys}:=
\begin{cases}
\{\zeta\in \mathbb C\mid \Im\zeta >0\}&\text{for } 0\le \delta_\Gamma\le \|\rho\|\\
\{\zeta\in \mathbb C\mid \Im\zeta >\delta_\Gamma -\|\rho\|\}&\text{for } \|\rho\|< \delta_\Gamma\le 2\|\rho\|.
\end{cases}\label{eq:physicalhalfplaneGamma}
\eq

\begin{figure}[htb]
\begin{tikzpicture}[scale =1]
\fill[green!5!white] (-1,0) rectangle (4,4);
	\node[anchor= north west] at (3,4) {\(\zeta\)};	
	\node[anchor= south east] at (0.5,0) {\(\textcolor{green}{\mathcal H_\mathrm{phys}}\)};	
	
\fill[blue!5!white] (-1,2) rectangle (4,4);
	\node[anchor= north west] at (3,4) {\(\zeta\)};	
	\node[anchor= south east] at (0.5,2) {\(\textcolor{blue}{\mathcal H_\mathrm{phys}^\Gamma}\)};	
	
	\coordinate (P) at (1,0);
	\coordinate (S1) at (1,0.4);
	\coordinate (S2) at (1,1);
	\coordinate (S3) at (1,1.5);
	\coordinate (R) at (1,2);
	\coordinate (P') at (8,0);
	\coordinate (R') at (7,0);
	\coordinate (S1') at (7.2,0);
	\coordinate (S2') at (7.6,0);
	\coordinate (S3') at (7.8,0);
	
	\coordinate (sP) at (11,2);
	\coordinate (sP+) at (11,0);
	\coordinate (sP-) at (4,-0.2);
	
	\coordinate (E1) at (1,-0.8);
	\coordinate (E2) at (1,-1.6);
	\coordinate (sO) at (6.1,0);

	\draw(1,-2) -- (1,4) node[shift={(0.2,-0.2)}]{\(\)};
	\draw[blue](-1,2) -- (4,2);
	\draw[green](-1,0) -- (4,0);
	\node at (P) {\(\textcolor{green}{\bullet}\)} node[shift={(1.2,-0.2)}]{\(0\)};
	\node at (R) {\(\textcolor{green}{\bullet}\)} node[shift={(2.05,1.75)}]{\(i(\delta_\Gamma -\|\rho\|)\)};
	\node at (S1) {\tiny\(\textcolor{green}{\bullet}\)};
	\node at (S2) {\tiny\(\textcolor{green}{\small\bullet}\)};
	\node at (S3) {\tiny\(\textcolor{green}{\small\bullet}\)};
	\node at (E1) {\tiny\(\textcolor{red}{\mathrm x}\)};
	\node at (E2) {\tiny\(\textcolor{red}{\mathrm x}\)};

\fill[blue!5!white] (6,-1) rectangle (11,4);
	\node[anchor= north west] at (8,4) {\(z=\|\rho\|^2+\zeta^2\)};	
	\node[anchor= south east] at (0.5,2) {\(\)};

	\begin{scope}[shift=(R'), rotate=-90]
	
	\draw[blue,fill=green!5!white](0,0) parabola (1,1) ;

	\draw[blue,fill=green!5!white](0,0) parabola (-2,4) ;
     \end{scope}

   \fill[green!5!white] (R') -- (sP) -- (sP+) -- (R');
   \fill[green!5!white] (R') -- (sP+) -- (11,-1) -- (8,-1) -- (R');
     
    \draw(8,0) -- (6,0);
    \draw[green,thick](8,0) -- (11,0);
     
    \node at (R') {\(\textcolor{green}{\bullet}\)}; 
    \node at (P') {\(\textcolor{green}{\bullet}\)} node[shift={(8.4,-0.3)}]{\(\|\rho\|^2\)};

    \node at (S1') {\tiny\(\textcolor{green}{\bullet}\)};
	\node at (S2') {\tiny\(\textcolor{green}{\small\bullet}\)};
	\node at (S3') {\tiny\(\textcolor{green}{\small\bullet}\)};


	
	
\end{tikzpicture}
	
	\end{figure}

Denote by $r_\zeta\in \D'(X\times X)$ and  $r^\Gamma_\zeta\in \D'((\gam X)\times (\gam X))$ the Schwartz kernels of $R_\zeta$ and $R^\Gamma_\zeta$, respectively, and by 
$
\tilde{r}^\Gamma_\zeta\in \D'(X\times X)^{\Gamma\times \Gamma}
$ 
the distribution corresponding to $r^\Gamma_\zeta$ when identifying $\D'((\gam X)\times (\gam X))=\D'(X\times X)^{\Gamma\times \Gamma}$. Note here that the identification uses the pullback via the canonical projection $\mathrm{pr}_\Gamma:X\to\Gamma\backslash X$, which is a submersion (cf. \cite[Thm.~6.12]{H90}).

When $\zeta$ lies sufficiently deep inside the physical half-plane, $\tilde r_\zeta^\Gamma$ is nothing but the ``$\Gamma$-summation'' (or ``$\Gamma$-average'') of $r_\zeta$. To make this precise we formulate the following lemma.
\begin{lem}[c.f.\ {\cite[Lemma 9]{BO12}}]\label{lem:rzetasummation}
For all $\zeta\in \mathcal{H}^\Gamma_\mathrm{phys}$ satisfying $\Im \zeta> \sqrt{(\Re \zeta)^2 +c_\Gamma}$, one has
\bq
\tilde r_\zeta^\Gamma=\sum_{\gamma \in \Gamma}(e,\gamma)^\ast r_\zeta,\label{eq:rzetasummation}
\eq
where
\bq
c_\Gamma:=\delta_\Gamma^2+\max(0,\delta_\Gamma-\norm{\rho})^2\in [\delta_\Gamma^2,2\delta_\Gamma^2]\label{eq:cGamma}
\eq
and $(\gamma_1,\gamma_2)^\ast$ denotes the pullback action of an element $(\gamma_1,\gamma_2)\in \Gamma\times \Gamma$. In particular, the right-hand side of \eqref{eq:rzetasummation} is a well-defined distribution on $X\times X$. 
\end{lem}

\begin{proof}
We argue as in \cite[proof of Lemma 9]{BO12}, adding (and repeating) some details. Let $\mathrm{diag}_X\subset X\times X$ be the diagonal and define $$D_\Gamma:=(\{e\}\times \Gamma)\cdot \mathrm{diag}_X=(\Gamma\times\{e\})\cdot \mathrm{diag}_X\subset X\times X.$$
The proper discontinuity of the $\Gamma$-action on $X$ implies that $D_\Gamma$ is a closed subset of $X\times X$ (c.f.\ \cite[Thm.\ 11 (6)]{kapovich23}). 
As pointed out in \cite[p.\ 137]{BO12}, the distribution $r_\zeta$ is smooth outside the diagonal, which implies that on  $X\times X\setminus D_\Gamma$ all the distributions $(e,\gamma)^\ast r_\zeta$, $\gamma \in \Gamma$, are smooth. Now, let $f\in \CT(X\times X)$ and let $\Gamma=\{\gamma_1,\gamma_2,\ldots\}$ be an enumeration of the group $\Gamma$. Then we have for each $N\in \N$
\begin{align*}
&\sum_{n=1}^N \big|((e,\gamma_n)^\ast r_\zeta)(f)\big|=\sum_{\substack{n\leq N:\\  (\supp f)\cap (e,\gamma_n^{-1})\cdot(\supp f)=\emptyset}} \big|((e,\gamma_n)^\ast r_\zeta)(f)\big|\\
&\qquad\qquad\qquad\qquad\quad+\sum_{\substack{n\leq N:\\  (\supp f)\cap (e,\gamma_n^{-1})\cdot(\supp f)\neq\emptyset}} \big|((e,\gamma_n)^\ast r_\zeta)(f)\big|\\
&\leq \sum_{\substack{n\leq N:\\  (\supp f)\cap (e,\gamma_n^{-1})\cdot(\supp f)=\emptyset}}\hspace*{-2em} \big|((e,\gamma_n)^\ast r_\zeta)(f)\big|\;+\sum_{\substack{\gamma\in \Gamma:\\ (\supp f)\cap (e,\gamma)\cdot(\supp f)\neq\emptyset}}\hspace*{-2em} \big|((e,\gamma)^\ast r_\zeta)(f)\big|.
\end{align*}
Here the last term is a finite sum because the $(\{e\}\times \Gamma)$-action on $X\times X$ is properly discontinuous. On the other hand, for each $\gamma \in \Gamma$ such that $(\supp f)\cap (e,\gamma^{-1})\cdot(\supp f)=\emptyset$, we can write
\begin{align*}
\big|((e,\gamma)^\ast r_\zeta)(f)\big|&=\big|((e,\gamma)^\ast r_\zeta)(f|_{X\times X\setminus D_\Gamma})\big|\\
&=\Big| \int_{X\times X \setminus D_\Gamma}((e,\gamma)^\ast r_\zeta)(x,y)\overline{f(x,y)}\,\d x \d y\Big|
\end{align*}
and then apply \cite[Lemma 8]{BO12}, as well as the expression for the minimum of the spectrum of $\Delta_{\gam X}$ stated above, to estimate for every compact set $W\subset \mathcal{H}^\Gamma_\mathrm{phys}$:
\begin{align*}
 \forall\; \zeta\in W:  \quad
 \big|((e,\gamma)^\ast r_\zeta)(f)\big|&\leq C \norm{f}_\infty e^{-a\sqrt{-\max(0,\delta_\Gamma-\norm{\rho})^2-\Re(\zeta^2)}}
\end{align*}
with a constant $C>0$ depending on $W$ and $\supp f$, and the exponent
$$
a:=\inf_{(x,y)\in \supp f}d_X(x,\gamma y)>0.
$$
Now, in view of the definition of the critical exponent $\delta_\Gamma$, the condition for convergence as $N\to \infty$ becomes 
$$
\sqrt{-\max(0,\delta_\Gamma-\norm{\rho})^2-\Re(\zeta^2)}>\delta_\Gamma,
$$ 
which is equivalent to 
$$
(\Im \zeta)^2>(\Re\zeta)^2+\underbrace{\delta_\Gamma^2+\max(0,\delta_\Gamma-\norm{\rho})^2}_{=c_\Gamma}.
$$
The convergence estimate depends on $f$ only via $\supp f$ and $\norm{f}_\infty$, so that we obtain distributional convergence of the right-hand side of \eqref{eq:rzetasummation} for all $\zeta\in \mathcal{H}_\mathrm{phys}$ with $\Im \zeta>\sqrt{(\Re \zeta)^2 +c_\Gamma}$.

It remains to show that the right-hand side of \eqref{eq:rzetasummation}  equals $\tilde r_\zeta^\Gamma$, which can be done exactly as in the end of the proof of \cite[Lemma 9]{BO12}.
\end{proof}

\begin{rem}In \eqref{eq:rzetasummation} one could equivalently take $\sum_{\gamma \in \Gamma}(\gamma,e)^\ast r_\zeta$: since $\Delta_X$ is $G$-equivariant, so is each of its resolvents $R_\zeta$ for $\zeta\in \mathcal{H}_\mathrm{phys}$. Thus $r_\zeta$ satisfies $(g,g^{-1})\cdot r_\zeta=r_\zeta$ for all $g \in G$.
\end{rem}

By \cite[Lemma 6]{BO12}, $r_\zeta$ and $r_\zeta^\Gamma$ (equivalently $\tilde r_\zeta^\Gamma$) extend to $\C$ as meromorphic families of distributions.  
The meromorphic families of operators $\CT(X)\to \D'(X)$ defined by the distribution kernels $r_\zeta$ and $\tilde r_\zeta^\Gamma$ will be denoted by $R_\zeta$ and $\tilde R_\zeta^\Gamma$, respectively (where $\tilde R_\zeta^\Gamma$ maps actually into $\D'(X)^\Gamma$ thanks to the $\Gamma\times \Gamma$-invariance of $\tilde r_\zeta^\Gamma$), while we denote the meromorphic family of operators $\CT(\gam X)\to \D'(\gam X)$ with distribution kernels $r_\zeta^\Gamma$ by $R_\zeta^\Gamma$. This is compatible with the previous notation in the sense that $R_\zeta$ and $R_\zeta^\Gamma$ are meromorphic extensions of the operator families from \eqref{eq:RL2} with domains shrinked to compactly supported smooth functions and codomains enlarged to distributions.

\section{Scattering matrices}\label{sec:scatteringmatrices}

In this section we keep the setup from Section~\ref{sec:ressccc} and follow \cite{BO00} in the definition of  $\lambda\mapsto S_\lambda^\Gamma$ as a meromorphic family of the scattering matrices for $\Gamma\backslash X$. It is our goal to relate the poles of this family, which we call \emph{scattering poles}, to the resonances of $\Gamma\backslash X$, i.e. the poles of the meromorphically continued family $\zeta\mapsto R^\Gamma_\zeta$ of resolvents.  

\subsection{Spherical principal series representations}\label{sec:princseries}

In order to introduce the scattering matrices we have to review some facts about spherical principal series representations. We warn the reader that there is no fixed standard notation in the literature. As we will have to use spectral information from different sources, we will have to keep track of the different conventions (see Remark~\ref{rem:Conventions1} below).

For $\lambda\in\aL_\C^*$ we define the one-dimensional group representation
\bq
\sigma_\lambda: P=MAN\to \C\setminus \{0\},\quad man\mapsto a^{\rho+\lambda}\label{eq:sigmalambda}
\eq
and the associated homogeneous line bundle 
\[
V(\sigma_\lambda):=G\times_P \C\to G/P, \quad [g,z]\mapsto gP,
\]
where $[g,z]$ is the $P$-orbit of $(g,z)\in G\times \C$ with respect to the right $P$-action given by $(g,z)\cdot p=(gp,\sigma_\lambda(p^{-1})z)$. Note that $G$ acts on the total space of $V(\sigma_\lambda)$ from the left. The  space $C^\infty(G/P;V(\sigma_\lambda))$ of all smooth sections $s:G/P\to V(\sigma_\lambda)$ can be identified with
\begin{align}
H^{\lambda,\infty}&:=\{f\in \Cinft(G)\mid\forall g\in G, p\in P: f(gp)=\sigma_\lambda(p)^{-1}f(g)\}\label{eq:Hsigmalambda}
\end{align}
so that $f\in H^{\lambda,\infty}$ corresponds to the section $s$ of $V(\sigma_\lambda)$ with $s(gP)=[g,f(g)]$. 

The left-regular re\-pre\-sen\-ta\-tion of $G$ on $H^{\lambda,\infty}$ will be denoted by $\pi^{\lambda,\infty}$:
\bq
\big(\pi^{\lambda,\infty}(g)f\big)(g'):= f(g^{-1}g'),\qquad f\in H^{\lambda,\infty}.\label{eq:pilambdasigma}
\eq
It is called (spherical, smooth) principal series representation with parameter $\lambda$. 

Due to the Iwasawa decomposition the restriction map
\begin{align}\begin{split}
H^{\lambda,\infty}&\to \Cinft(K)^M:=\{f\in \Cinft(K) \mid \forall k\in K,m\in M: f(km)=f(k)\}\label{eq:Hlambdaiso}\\
f &\mapsto f|_K\end{split}
\end{align}
is an isomorphism of vector spaces. This allows us to consider $H^{\lambda,\infty}$ and $C^\infty(G/P;V(\sigma_\lambda))$ as $\lambda$-independent vector spaces. It also shows that $V(\sigma_\lambda)$ is in fact a trivial line bundle but we will soon pass to its $\Gamma$-quotient, which  no longer needs to be trivial. 

We equip the vector space $C^{\infty}(G/P;V(\sigma_{\lambda}))\cong H^{\lambda,\infty}$ with the Fr\'echet topology such that  \eqref{eq:Hlambdaiso} is an isomorphism of topological vector spaces, where $\Cinft(K)^M$ is equipped with the subspace topology inherited from the standard Fr\'echet topology on $\Cinft(K)$. This way, we can consider $H^{\lambda,\infty}$ and $C^\infty(G/P;V(\sigma_\lambda))$ as $\lambda$-independent Fr\'echet spaces.

By restricting the line bundle  $V(\sigma_\lambda)$ to $\Omega\subset G/P=\partial X$ and passing to the $\Gamma$-quotient, we obtain the locally homogeneous  vector bundle
\[
V_{\Gamma\backslash \Omega}(\sigma_{\lambda}):=\gam V(\sigma_\lambda)|_\Omega
\]
over the smooth manifold $\Gamma\backslash \Omega$. The vector space $C^\infty(\Gamma\backslash \Omega;V_{\Gamma\backslash \Omega}(\sigma_\lambda))$ corresponds canonically to the space of $\Gamma$-invariant smooth sections of $V(\sigma_\lambda)|_\Omega$, which in turn corresponds canonically to a closed subspace of left-$\Gamma$-invariant functions in $H^{\lambda,\infty}$.  
 Thus, the space $C^\infty(\Gamma\backslash \Omega;V_{\Gamma\backslash \Omega}(\sigma_\lambda))$ inherits a Fr\'echet topology from $H^{\lambda,\infty}$. In particular, we can consider $C^\infty(\Gamma\backslash \Omega;V_{\Gamma\backslash \Omega}(\sigma_\lambda))$  as a  $\lambda$-independent Fr\'echet space, too. 

 \begin{defn}\label{def:distribsec}We call the topological vector space
\[
C^{-\infty}(G/P;V(\sigma_\lambda)):=C^\infty(G/P;V(\sigma_{-\lambda}))',
\] 
where ``prime'' denotes the topological dual, the space of \emph{distributional sections} of $V(\sigma_\lambda)$. 
\end{defn}
We equip $C^{-\infty}(G/P;V(\sigma_\lambda))$ with the dual $G$-representation of $\pi^{-\lambda,\infty}$, which we denote by $\pi^{\lambda,-\infty}$, and write the subspace of $\Gamma$-invariant elements as  ${}^\Gamma C^{-\infty}(G/P;V(\sigma_\lambda))$. The map
 \begin{align}
\begin{split}
 C^{\infty}(G/P;V(\sigma_\lambda))\cong H^{\lambda,\infty}&\to C^{-\infty}(G/P;V(\sigma_\lambda)),\\
 f&\mapsto \eklm{f|_K,\cdot}_{\L^2(K)}\label{eq:injCinftDistr}\end{split}
 \end{align}
 is well-defined, injective (since the restriction \eqref{eq:Hlambdaiso} is injective), and continuous. Moreover,  the injection \eqref{eq:injCinftDistr} is $G$-equivariant; for this, the  sign change of $\lambda$ in Definition \ref{def:distribsec} is essential.   
By virtue of this equivariant injection we can view the principal series representation $(\pi^{\lambda,\infty},C^{\infty}(G/P;V(\sigma_\lambda)))$ as a subrepresentation of $(\pi^{\lambda,-\infty},C^{-\infty}(G/P;V(\sigma_\lambda)))$. We  call the latter  \emph{distributional principal series representation} with parameter $\lambda$.

\subsection{Comparing conventions}

We will have to keep track of the conventions for meromorphic continuations with respect to complex parameters used in the various sources we rely on. To simplify this task we include a couple of remarks containing the pertinent comparisons.

\begin{rem}[Signs and left/right regular representation]\label{rem:Conventions1} Our definitions \eqref{eq:sigmalambda}, \eqref{eq:Hsigmalambda} differ from the ones in \cite[Section 3]{BO00}, \cite[\S~2.2]{Ol} and \cite[\S~2]{HHP19} by the sign of $\lambda$. We chose this sign convention to be consistent with \cite[VII, §1]{Kna86}. This is also consistent with the sign convention in \cite[Chap.~3]{VoganUnitaryRep}. In all these references the $G$-action is the left regular action. 
On the other hand, \cite[\S~10.1]{W92} and \cite{Knapp-Stein} use the right regular version of the principal series representations. Note that the map $f\mapsto \check{f}$ with $\check f(g):=f(g^{-1})$ intertwines left and right regular actions. Up to this intertwiner the sign conventions in  \cite[\S~10.1]{W92} and \cite{Knapp-Stein} agree with the sign conventions in \cite[VII, §1]{Kna86}.  Finally, we point out that our topologies on  $C^{\infty}(G/P;V(\sigma_{\lambda}))$ and $C^\infty(\Gamma\backslash \Omega;V_{\Gamma\backslash \Omega}(\sigma_\lambda))$ are equivalent to those of \cite[p.~89]{BO00}.
\end{rem}

\begin{rem}[Norms and identifications]\label{rem:comparing conventions_neu} Recall from Section \ref{sec:ressccc} that we fixed an inner product $\eklm{\cdot,\cdot}$ on $\g$. In \cite[\S~6.1]{BO12} there is the freedom to choose a bilinear form $b$ inducing the Riemannian metric on $G/K$, and we let $b$ be given by $\eklm{\cdot,\cdot}$. Then the norms on $\g$ and $\g^\ast\cong \g$ (using the isomorphism provided by $\eklm{\cdot,\cdot}$) used here, in \cite[\S~6]{BO12}, and in  \cite{HHP19} all agree. 

We use the convention from \cite[\S~6.1]{BO12} that the complex vector space $\aL^\ast_\C$ is identified with $\C$ by complex linear extension of the unique order-preserving isometric isomorphism $\aL^\ast\cong \R$. This means that $\rho\in \aL^\ast$ is identified with $\norm{\rho}\in \R$ (which accordingly is written simply as $\rho$ in \cite[\S~6]{BO12}), and a general a number $z\in \C$ corresponds to the element
\bq
\frac{z}{\norm{\rho}}\rho\in \aL^\ast_\C.\label{eq:lambdaz}
\eq
This identification differs slightly from that in \cite{HHP19}, where $\rho\in \aL^\ast$ gets identified with $\frac{\norm{\rho}}{\norm{\alpha_0}}\in \R$, $\alpha_0\in \Sigma^+\subset \aL^\ast$ being the unique reduced positive root, and consequently a general number $z\in \C$ corresponds to $\norm{\alpha_0}$ times the element \eqref{eq:lambdaz}.  
\end{rem}

\begin{rem}[{Branched covers, physical sheets and half-planes}]\label{rem:Conventions3}Our choice of the physical half-plane $\mathcal{H}_\mathrm{phys}=\{\zeta\in \C\,|\, \Im \zeta >0\}$ is compatible with \cite{HHP19} when taking into account the sign change mentioned in Remark \ref{rem:Conventions1}. On the other hand, Bunke and Olbrich work in \cite[p.\ 135]{BO12} with the right half-plane $\mathcal{H}_\mathrm{phys,r}^{\mathrm{BO}}:=\{\lambda\in \C\,|\,\Re \lambda >0\}$ but their argument works just as well with the left half-plane $\mathcal{H}_\mathrm{phys,l}^{\mathrm{BO}}:=\{\lambda\in \C\,|\,\Re \lambda <0\}$.

Let ${\widetilde \C}^\pm$ be the branched covers of $\C$ on which the functions 
\bq
{\widetilde \C}^\pm\owns z\mapsto \sqrt{\pm(\norm{\rho}^2-z)}=:\zeta_\pm(z)\in \C\label{eq:pmfcns}
\eq are holomorphic, respectively.  In \cite[p.\ 135]{BO12} Bunke and Olbrich work with the function $\zeta_+$, while Hansen, Hilgert and Parthasarathy work with $\zeta_-$ in  \cite{HHP19}.  
We call $\widetilde\C^-_\mathrm{phys}:=\zeta_-^{-1}(\mathcal{H}_\mathrm{phys})\subset \widetilde \C^-$ the \emph{physical sheet} of $\widetilde \C^-$ and  $\widetilde\C^+_\mathrm{phys,r}:=\zeta_+^{-1}(\mathcal{H}_\mathrm{phys,r}^{\mathrm{BO}}),\widetilde\C^+_\mathrm{phys,l}:=\zeta_+^{-1}(\mathcal{H}_\mathrm{phys,l}^{\mathrm{BO}})\subset \widetilde \C^+$ the \emph{positive/negative physical sheets} of $\widetilde \C^+$, respectively.  We emphasize that there is no preferred choice of the physical sheet in \cite[Section 6]{BO12}, so we can use their results for parameters in $\mathcal{H}_\mathrm{phys,r}^{\mathrm{BO}}$ as well as in $\mathcal{H}_\mathrm{phys,l}^{\mathrm{BO}}$. 

As before, we will denote the complex numbers in our physical half-plane by $\zeta$, while we will write $\lambda$ for the numbers in $\mathcal{H}_\mathrm{phys,r}^{\mathrm{BO}}$ and $\mathcal{H}_\mathrm{phys,l}^{\mathrm{BO}}$. We have obvious isomorphisms $\mathcal{H}_\mathrm{phys}\cong \mathcal{H}_\mathrm{phys,r}^{\mathrm{BO}}$, $\mathcal{H}_\mathrm{phys}\cong \mathcal{H}_\mathrm{phys,l}^{\mathrm{BO}}$ given by
\bq
\mathcal{H}_\mathrm{phys}\owns \zeta\mapsto \lambda=\begin{cases}-i\zeta &\in \mathcal{H}_\mathrm{phys,r}^{\mathrm{BO}},\\
i\zeta&\in \mathcal{H}_\mathrm{phys,l}^{\mathrm{BO}}.\end{cases}\label{eq:zetaoflambda}
\eq
\end{rem}

\subsection{Scattering matrices} We build on the descriptions of scattering matrices given in \cite{BO00} and \cite{HHP19}.

Then, according to \cite[pp. 79--80]{BO00}, using their isometric identification $\C\owns \lambda\leftrightarrow \lambda\frac{\rho}{\norm{\rho}}\in \aL_\C^\ast$, we have the three  meromorphic families on $\C$ of maps
\begin{eqnarray}
\mathrm{ext}_\lambda&:&C^{-\infty}(\Gamma\backslash \Omega, V_{\Gamma\backslash \Omega}(\sigma_\lambda))
\to {}^\Gamma C^{-\infty}(\partial X, V(\sigma_\lambda))\\ 
\mathrm{res}_\lambda&:&{}^\Gamma C^{-\infty}(\partial X, V(\sigma_\lambda))\mapsto C^{-\infty}(\Gamma\backslash \Omega, V_{\Gamma\backslash \Omega}(\sigma_\lambda))\\ 
\hat J_\lambda&:&C^{-\infty}(\partial X, V(\sigma_\lambda))
\to C^{-\infty}(\partial X, V(\sigma_\lambda)), 
\end{eqnarray}
where the latter is the ``unnormalized'' Knapp-Stein intertwiner described in detail in \cite[Thm.~3]{Knapp-Stein} (in which the identification $\C\owns \lambda\leftrightarrow \lambda\rho\in \aL_\C^\ast$ is used, so that our $\lambda$ corresponds to $\frac{\lambda}{\norm{\rho}}$ in \cite[Thm.~3]{Knapp-Stein}). Recall from Remark \ref{rem:Conventions1} that the $P$-representation $\sigma_\lambda$ defined in \eqref{eq:sigmalambda} is the same as in \cite{Knapp-Stein} but corresponds to $\sigma_{-\lambda}$ in \cite{BO00}. 
The Harish-Chandra $c$-function is the meromorphic function on $\aL^\ast_\C$ satisfying
\[
c(\lambda\rho):=\int_{\bar N}e^{-(\lambda+1)\rho(H(\bar n))}\d \bar n,\qquad \forall\; \lambda \in \C,\;\Re\lambda >0.
\]
In view of the chosen identification of $\aL_\C^\ast$ with $\C$ and the sign difference in the conventions for the principal series parameter used here and in \cite{BO00}, respectively, it is convenient to consider the meromorphic function $c_\C$ on $\C$ defined by
\[
c_\C(\lambda):=c\Big(\frac{-\lambda}{\norm{\rho}}\rho\Big).
\] 
Then we define the normalized  Knapp-Stein intertwiner
\[
J_\lambda:=\frac{1}{c_\C(\lambda)} \hat J_\lambda.
\]
Note that our $J_\lambda$ corresponds to $J_{\mathbf{1},-\lambda}$ in \cite[p.\ 101]{BO00}. 

Recall the critical exponent $\delta_\Gamma$ from \eqref{eq:criticalexponent} and note that if $\delta_\Gamma^{\mathrm{BO}}$ denotes the critical exponent from \cite[Def.\ 2.2]{BO00}, then $\delta_\Gamma=\delta_\Gamma^{\mathrm{BO}}+\norm{\rho}$. 

\begin{defn}[Scattering matrices, {\cite[Def.~5.6]{BO00}}] 
Let $\mathrm{Re}(\lambda)<\norm{\rho}-\delta_\Gamma$. Then $S_\lambda^\Gamma:= \mathrm{res}_\lambda\circ J_\lambda\circ \mathrm{ext}_\lambda$ is called the $\lambda$-\emph{scattering matrix} for $\Gamma\backslash X$. 
\end{defn}

Note that our $S_\lambda^\Gamma$ corresponds to $S_{-\lambda}$ in \cite[Def.~5.6]{BO00}.

From \cite[Lemma~5.7]{BO00} we know that for $\mathrm{Re}(\lambda)<\norm{\rho}-\delta_\Gamma$ the family 
\[\lambda\mapsto S_\lambda^\Gamma: C^{\pm\infty}(\Gamma\backslash \Omega, V_{\Gamma\backslash \Omega}(\sigma_\lambda))\to C^{\pm\infty}(\Gamma\backslash \Omega, V_{\Gamma\backslash \Omega}(\sigma_{-\lambda}))
\]
is meromorphic. More precisely, we have the following result. 
\begin{prop}[Scattering matrices, {\cite[Thm.~5.10]{BO00}}]\label{prop:BOext} Suppose that $G\not=F_4^{-20}$. Then the families 
\[\lambda\mapsto S_\lambda^\Gamma: C^{\pm\infty}(\Gamma\backslash \Omega, V_{\Gamma\backslash \Omega}(\sigma_\lambda))\to C^{\pm\infty}(\Gamma\backslash \Omega, V_{\Gamma\backslash \Omega}(\sigma_{-\lambda}))
\]
and 
\[\lambda\mapsto \mathrm{ext}_\lambda: C^{-\infty}(\Gamma\backslash \Omega, V_{\Gamma\backslash \Omega}(\sigma_\lambda))\to  {}^\Gamma C^{-\infty}(\partial X, V(\sigma_\lambda)),
\]
originally defined for $\mathrm{Re}(\lambda)<\norm{\rho}-\delta_\Gamma$, have meromorphic continuations to all of $\C$. The family $\mathrm{ext}$ has at most finite-dimensional singularities. Moreover, the following identities hold:
\begin{eqnarray}
\mathrm{ext}_\lambda&=&J_{-\lambda}\circ \mathrm{ext}_\lambda \circ S^\Gamma_\lambda,\\
S^\Gamma_{-\lambda}\circ S^\Gamma_\lambda&=&\mathrm{id}.
\end{eqnarray}  
\end{prop}

We expect that the necessity of the hypothesis $G\not=F_4^{-20}$ in Proposition~\ref{prop:BOext} is an artefact of its proof. Recall that the use of this result is the only reason why we had to restrict our attention to classical groups. We will address this issue elsewhere.

\begin{defn}[Scattering matrices for the locally and globally symmetric space, c.f.\ {\cite[\S 7]{HHP19}}] \label{def:Sop}
We call the two meromorphic operator families
\begin{align*}
\mathbf{S}_\zeta&:=J_{i\zeta},\\
\mathbf{S}_\zeta^\Gamma &:=S^\Gamma_{i\zeta}
\end{align*}
the $\zeta$-\emph{scattering matrices} for $X$.
\end{defn}

Note that our $\mathbf{S}_\zeta$ corresponds to $S_{-\frac{\zeta\alpha_0}{\norm{\alpha_0}}}$ in \cite[\S 7]{HHP19}. This takes into account that our principal series representation parameter $\lambda$ corresponds to $-\lambda$ in  \cite{HHP19}, while our resolvent $R_\zeta$ corresponds to $R_{-\frac{\zeta\alpha_0} {\norm{\alpha_0}}}=R_{\frac{\zeta\alpha_0} {\norm{\alpha_0}}}$ in \cite{HHP19}.

\section{Relating scattering matrices and resolvent kernels}

The key to the comparison between scattering matrices and resolvent kernels for \emph{global} symmetric spaces given in \cite{HHP19} was the possibility to take boundary values of Laplace eigenfunctions on $X$. The boundary value maps are given by the inverses of Poisson transforms. Moreover, the proof of \cite[Lemma~11]{BO12} suggests a way how to use the knowledge about the relation between resolvent poles and scattering poles of rank one global symmetric spaces in order to determine the corresponding relation for rank one \emph{locally} symmetric spaces. We will follow this strategy and start by reviewing a few facts about Poisson transforms and boundary value maps.

\subsection{Poisson transforms and boundary values}\label{subsec:Poisson trafo}

We define the \emph{Poisson transform} with parameter $\lambda\in \aL_\C^\ast$ by
\begin{align*}
\mathcal P_{\lambda}: C^{-\infty}(G/P;V(\sigma_{\lambda}))&\to C^\infty(G/K),\\
\mathcal P_{\lambda}(u)(g)&:= (\pi^{\lambda,-\infty}(g^{-1})u)(1_{-\lambda}),
\end{align*}
where $1_{-\lambda}\in C^{\infty}(G/P;V(\sigma_{-\lambda}))\cong H^{-\lambda,\infty}$  corresponds to the constant function on $K$ with value $1$ via the isomorphism \eqref{eq:Hlambdaiso}. If $f\in H^{\lambda,\infty}\cong C^{\infty}(G/P;V(\sigma_\lambda))$ is identified with an element of $C^{-\infty}(G/P;V(\sigma_{\lambda}))$ via the injection \eqref{eq:injCinftDistr}, then 
\bq
\mathcal P_{\lambda}(f)(g)=\int_K f(gk)\d k=\int_K e^{(\lambda-\rho)(H(g^{-1}k))} f(k)\d k. \label{eq:Poissonforfunctions}
\eq

\begin{rem}[Comparing conventions]\label{rem:Conventions2}
Note that due to our sign convention, see Remark~\ref{rem:Conventions1}, our Poisson  transform  $\mathcal P_{\lambda}$ corresponds to $P_{-\lambda}^T$ in \cite[\S~4.2]{BO12} and \cite[Def.~4.8]{BO00}, respectively to $P_{-\lambda}$ in \cite[(2)]{HHP19}. 
\end{rem}

 For generic parameters $\lambda$ the Poisson transformations are essentially inverted by  \emph{boundary value} maps. In \cite[(29)]{HHP19} this reads
\[\beta^{\mathrm{HHP}}_{\rho-\lambda}\circ P_\lambda =c(\lambda) \mathrm{Id},\]
where we decorated the boundary value map with HHP in order to avoid confusion  later on. 

We need to relate the boundary value maps  used in \cite[\S~6.3]{BO12} to the Poisson transforms used in \cite[\S~4.2]{BO12} in order to relate them to the boundary value maps $\beta_\lambda$ used in \cite[Lemma~6.1]{HHP19}. In \cite[\S~6.3]{BO12} the authors are not very precise about the construction of boundary values they are using.  They simply refer to \cite{KO77} and \cite{OS80}, but they give the set of leading exponents. We assume they are using the conventions of \cite{Ol}, which are compatible with \cite{BO00,BO12,HHP19}. Note that \cite[Thm.~4.17]{Ol} (see also \cite[Thm.~3.9]{Ol} and \cite[Lemma~9.28]{Kna86}) relates the corresponding boundary value maps with the Poisson transform in the following formula
\[\beta^{\mathrm{Olb94}}_{1,\lambda}\circ P^\gamma_{\lambda} = c_\gamma(\lambda)\, \mathrm{Id}.\]
The definitions of the $c$-functions in \cite{HHP19} and \cite{Ol} are compatible. In our situation $\sigma$ and $\gamma$ are trivial, which means that the boundary value map $\beta^{\mathrm{Olb94}}_{1,\lambda}$ from \cite{Ol} equals the boundary value map $\beta^{\mathrm{HHP}}_{\rho-\lambda}$ from \cite{HHP19}.   

In \cite[\S~6.3]{BO12} Bunke and Olbrich invoke \cite{KO77} to guarantee the existence of boundary values for eigenfunctions of the Laplacians on the two factors of $(X\times X)\setminus D_\Gamma$, see \cite[(48)]{BO12}. It is defined for parameters determined by the leading exponents of Casimir eigenfunctions (see \cite[Lemmas~10 \& 11]{BO12} and recall that in our situation $\sigma$ and $\tau$ are trivial).

In principle the construction of the boundary value maps is built on the study of asymptotic behavior of eigenfunctions. The construction in \cite{KO77} depends on a specific and very technical theory, the algebraic analysis initiated by Sato. When the boundary value map is applied to moderately growing eigenfunctions on all of $X\times X$ we can use \cite{vdBS87} to describe the boundary value maps. This description shows that the direct product situation leads to a decoupling of the asymptotics which leads to a decoupling of the boundary value map.  
Taking into account the notation and the various conventions from Remarks \ref{rem:Conventions2} and \ref{rem:Conventions3}, and using the first of the isomorphisms \eqref{eq:zetaoflambda},  the boundary value map from \cite[\S~6.3]{BO12} suitably restricted is given by
\bq
\zeta\in\mathcal{H}_\mathrm{phys}:\qquad
\beta_\nu=\beta_{\nu(\zeta)}:= \beta^{\mathrm{HHP}}_{\rho-i\frac{\zeta \alpha_0}{\norm{\alpha_0}}}\otimes \beta^{\mathrm{HHP}}_{\rho-i\frac{\zeta \alpha_0}{\norm{\alpha_0}}},\label{eq:boundarymaps}
\eq 
with $\nu\in (\mathfrak a_{\C}^*)^2\cong\C^2$ denoting the leading exponents of the eigensections one takes boundary values of. 

Recall the Schwartz kernels $r_\zeta$ and $\tilde r^\Gamma_\zeta$ of $R_\zeta$ and $\tilde R^\Gamma_\zeta$, respectively from Section~\ref{sec:resolventkernels}. Further, 
let $s_\zeta$ and $\tilde s^\Gamma_\zeta$  be the Schwartz kernels of $\mathbf{S}_\zeta$ and  the $\Gamma$-invariant lift of the Schwartz kernel of $\mathbf{S}^\Gamma_{\zeta}$,  respectively. We consider  all of these as meromorphically extended to $\C$. 

\begin{lem}\label{lem:bdyval-r-to-s} We have the identity of meromorphic functions 
\bq
s_{\zeta}|_{\partial X\times \partial X\setminus \mathrm{diag}_{\partial X}}=-2i\norm{\alpha_0}\zeta \, \beta_{\nu(\zeta)}(r_{\zeta}|_{X\times X\setminus \mathrm{diag}_X}).\label{eq:HHPformula}
\eq
\end{lem}

\begin{proof} Let $f,g\in \Cinft(\partial X)$ have disjoint supports. Then recalling \eqref{eq:boundarymaps} and applying \cite[last eq.\ on p.\ 20]{HHP19} with the parameter ${-\frac{\zeta\alpha_0}{\norm{\alpha_0}}}$ gives
\begin{align*}
\eklm{s_{\zeta},f\otimes g}&=\eklm{\mathbf{S}_\zeta f,g }\\
&=-2i\norm{\alpha_0}\zeta\big<\beta^{\mathrm{HHP}}_{\rho-i\frac{\zeta \alpha_0}{\norm{\alpha_0}}}(\beta^{\mathrm{HHP}}_{\rho-i\frac{\zeta \alpha_0}{\norm{\alpha_0}}}R_\zeta)'f,g\big>\\
&=-2i\norm{\alpha_0}\zeta\big<R_\zeta\beta^{\mathrm{HHP}'}_{\rho-i\frac{\zeta \alpha_0}{\norm{\alpha_0}}}f,\beta^{\mathrm{HHP}'}_{\rho-i\frac{\zeta \alpha_0}{\norm{\alpha_0}}}g\big>\\
&=-2i\norm{\alpha_0}\zeta \big<r_\zeta|_{X\times X\setminus \mathrm{diag}_X},(\beta^{\mathrm{HHP}'}_{\rho-i\frac{\zeta \alpha_0}{\norm{\alpha_0}}}\otimes\beta^{\mathrm{HHP}'}_{\rho-i\frac{\zeta \alpha_0}{\norm{\alpha_0}}})(f\otimes g)\big>\\
&=-2i\norm{\alpha_0}\zeta \big<r_\zeta|_{X\times X\setminus \mathrm{diag}_X},\beta_{\nu(\zeta)}'(f\otimes g)\big>\\
&=-2i\norm{\alpha_0}\zeta \big<\beta_{\nu(\zeta)}(r_\zeta|_{X\times X\setminus \mathrm{diag}_X}),f\otimes g\big>.
\end{align*}
The claim follows by density.
\end{proof}

\subsection{Local-global comparisons}

The proof of \cite[Lemma~11]{BO12} establishes the existence of a scalar-valued meromorphic function $\C\owns\zeta\mapsto a_\zeta$ satisfying the following formulas. 

\begin{eqnarray}
\beta_{\nu(\zeta)}(r_{\zeta}|_{X\times X\setminus \mathrm{diag}_X}) &=& a_\zeta s_{\zeta}|_{\partial X\times \partial X\setminus\mathrm{diag}_{\partial X}}\label{eq:betanu_rz}
\\
\beta_{\nu(\zeta)}((\tilde r^\Gamma_{\zeta}-r_{\zeta})|_{X\times X\setminus (D_\Gamma\setminus \mathrm{diag}_X)}) &=& a_\zeta (\tilde s^\Gamma_{\zeta}-s_{\zeta})|_{\Omega\times \Omega\setminus (D^\Omega_\Gamma\setminus\mathrm{diag}_\Omega)},
\label{eq:betanu_rzB}
\end{eqnarray}
where
\[
D^\Omega_\Gamma:=(\{e\}\times \Gamma)\cdot \mathrm{diag}_\Omega=(\Gamma\times\{e\})\cdot \mathrm{diag}_\Omega\subset\Omega\times \Omega.
\]
In other words, we also have
\begin{equation}\label{eq:Schwartz kernel key}
\beta_{\nu(\zeta)}(\tilde r^\Gamma_{\zeta}|_{X\times X\setminus D_\Gamma}) = a_\zeta \tilde s_{\zeta}^\Gamma|_{\Omega\times \Omega\setminus D^\Omega_\Gamma}.
\end{equation}
Actually, \cite[p.~140]{BO12} does only yield Equation~\eqref{eq:betanu_rzB}  on a neighborhood of $\Omega\times \Omega\setminus(D_\Gamma^\Omega\setminus\mathrm{diag}_\Omega)$ in the compactification of $X\times X$ (intersected with $X\times X\setminus\mathrm{diag}_X)$. But this suffices as the boundary values only depend on the values on a neighborhood of the boundary.

Lemma~\ref{lem:bdyval-r-to-s} shows that we know the function $a_\zeta$ explicitly on $(\partial X\times\partial X)\setminus \mathrm{diag}_{\partial X}$:
\bq
a_\zeta=\frac{i}{2 \norm{\alpha_0}\zeta}. \label{eq:azeta}
\eq
We thus arrive at the result
\bq
-2i \norm{\alpha_0}\zeta\, \beta_{\nu(\zeta)}(\tilde r^\Gamma_{\zeta}|_{X\times X\setminus D_\Gamma}) = \tilde s_{\zeta}^\Gamma|_{\Omega\times \Omega\setminus D^\Omega_\Gamma}.\label{eq:result111}
\eq

The following lemma will help us to remove $D_\Gamma$ and $D_\Gamma^\Omega$ from Equation~\eqref{eq:result111}.

\begin{lem}\label{lem:poles-diagonal}
Let $M$ be a smooth Riemannian manifold and $$P:\Cinft(M)\to \Cinft(M)$$ an elliptic differential operator of order $>0$ (i.e., non-constant). Suppose that the resolvent $$(P-z)^{-1}:L^2(M)\to L^2(M)$$ is defined for all $z$ in an open subset $U\subset \C$ and extends to a connected open subset $ V\supset U$ of $\C$ as a meromorphic family of operators
\[
R(z):\CT(M)\to \mathcal D'(M).
\]
Let $r(z)\in \mathcal D'(M\times M)$ be the meromorphic family of distributions given by the Schwartz kernel of $R(z)$. Then a complex number $z_0\in V$ is a pole of $r(z)$ if, and only if, it is a pole of the meromorphic family $$r(z)|_{M\times M\setminus \mathrm{diag}_M}\in \mathcal D'(M\times M\setminus \mathrm{diag}_M)$$ given by the restriction of $r(z)$ to the complement of the diagonal in $M\times M$.
\end{lem}
\begin{proof}
Denote by $P_{\CT}:\CT(M)\to \CT(M)$ and $P_{\mathcal D'}:\mathcal D'(M)\to \mathcal D'(M)$ the operators obtained from $P$ by restriction and duality, respectively. Then for all $z$ in the set $U$, where $R(z)$ is  $(P-z)^{-1}$ equipped with a smaller domain and a larger codomain, we have
\begin{align}
(P_{\mathcal D'}-z\, \mathrm{Id}_{\mathcal D'(M)})\circ R(z)&=I,\label{eq:id1111}\\
R(z)\circ (P_{\CT}-z\, \mathrm{Id}_{\CT(M)})&=I,\label{eq:id2222}
\end{align}
where  $I:\CT(M)\to\mathcal D'(M)$ 
is the inclusion. Now \eqref{eq:id1111} and \eqref{eq:id2222} are identities of holomorphic operator families on $U$  which continue meromorphically to the connected open set $V\supset U$, so  \eqref{eq:id1111} and \eqref{eq:id2222} also hold as identities of meromorphic operator families on $V$. 

Let $z_0\in V$. If $z_0$ is a pole of $r(z)|_{M\times M\setminus \mathrm{diag}_M}$, then it clearly is also a pole of $r(z)$. Conversely, assume that $z_0$ is a pole of $r(z)$. This means that for all $z\neq z_0$ in a neighborhood $U_0$ of $z_0$ in $V$ we have a unique presentation of $r(z)$ as a Laurent series
\bq
r(z)=\sum_{k\geq -k_0 }(z-z_0)^k a_k\label{eq:Laurent}
\eq
with distributions $a_k\in \mathcal D'(M\times M)$ and  $k_0\in \N=\{1,2,\ldots\}$ with $a_{-k_0}\neq 0$. Now suppose, in order to provoke a contradiction, that $z_0$ is not a pole of $r(z)|_{M\times M\setminus \mathrm{diag}_M}$. This means that for all $z$ in a neighborhood $U_1$ of $z_0$ in $V$ we have a unique presentation
\bq
r(z)|_{M\times M\setminus \mathrm{diag}_M}=\sum_{k\geq 0}(z-z_0)^k b_k\label{eq:pres1}
\eq
with distributions $b_k\in \mathcal D'(M\times M\setminus \mathrm{diag}_M)$. By restricting $r(z)$ to $M\times M\setminus \mathrm{diag}_M$ for $z\in U_0\cap U_1$, the uniqueness of the presentation \eqref{eq:pres1} implies that 
\bq
a_k|_{M\times M\setminus \mathrm{diag}_M}=0\qquad  \forall\; k<0. \label{eq:suppak}
\eq
For each $k\geq -k_0$, let $A_k:=\CT(M)\to \mathcal D'(M)$ be the operator with Schwartz kernel $a_k$. Then \eqref{eq:Laurent} implies that we have a convergent series
\bq
R(z)=\sum_{k\geq -k_0 }(z-z_0)^k A_k\label{eq:RzLaurent}
\eq
for all $z\in U_0$, and \eqref{eq:suppak} implies that $A_k$ is a local operator for all $k<0$: If $f\in \CT(M)$ and $\chi\in \CT(M\setminus \supp f)$, then \eqref{eq:suppak} gives
\[
A_k(f)(\chi)=a_k(f\otimes \chi)=0
\]
since $\supp (f\otimes \chi)\subset M\times M\setminus \mathrm{diag}_M$. Thus $\supp A_k(f)\subset \supp f$.

By \eqref{eq:RzLaurent} we can extract the operator $A_{-k_0}$ from the meromorphic family $R(z)$ as
\[
A_{-k_0}=\lim_{\substack{z\to z_0\\z\neq z_0}}(z-z_0)^{k_0}R(z).
\]
On the other hand, by \eqref{eq:id1111}, we have
\begin{align*}
(P_{\mathcal D'}-z\, \mathrm{Id}_{\mathcal D'(M)})\circ(z-z_0)^{k_0} R(z)&=
(z-z_0)^{k_0}(P_{\mathcal D'}-z\, \mathrm{Id}_{\mathcal D'(M)})\circ R(z)\\
&= (z-z_0)^{k_0}I\\
&\stackrel{z\to z_0}{\longrightarrow} 0.
\end{align*}
Using the equicontinuity of the operator family $P_{\mathcal D'}-z\, \mathrm{Id}_{\mathcal D'(M)}$ at $z=z_0$ we deduce that
\bq
(P_{\mathcal D'}-z_0\, \mathrm{Id}_{\mathcal D'(M)})\circ A_{-k_0}=0,\label{eq:pzakzero}
\eq
i.e., 
\bq
\mathrm{im}(A_{-k_0})\subset \ker(P_{\mathcal D'}-z_0\, \mathrm{Id}_{\mathcal D'(M)}).\label{eq:imker}
\eq
Now $P$ is elliptic, so \eqref{eq:imker} gives $\mathrm{im}(A_{-k_0})\subset \Cinft(M)$ by elliptic regularity. We already know that $A_{-k_0}$ is a local operator, so by Peetre's Theorem we conclude that $A_{-k_0}$ is a differential operator. In particular, $\mathrm{im}(A_{-k_0})\subset \CT(M)$. Knowing this, \eqref{eq:pzakzero} implies  that
\bq
(P_{\CT}-z_0\, \mathrm{Id}_{\CT(M)})\circ A_{-k_0}=0.\label{eq:compzero}
\eq
Now, around each $x\in M$ we can choose a neighborhood $W$ of $x$ in $M$ such that the differential operator $A_{-k_0}$ has a constant order $N$ on $W$.  Let $\sigma_P,\sigma_{A_{-k_0}}\in \Cinft(T^\ast M)$ be the principal symbols of $P$ and $A_{-k_0}$, respectively. By the ellipticity of $P$ we have $\sigma_P(x,\xi)\neq 0$ for all $\xi\neq 0$.  Since $P$ has order $>0$ it follows that $\sigma_P(x,\xi)-z_0\neq 0$ for all $\xi$ with $\norm{\xi}$ sufficiently large. From the composition theorem and \eqref{eq:compzero} we get
\[
(\sigma_P(x,\xi)-z_0)\sigma_{A_{-k_0}}=0,
\]
which implies that $\sigma_{A_{-k_0}}(x,\xi)=0$ for all $\xi$ with $\norm{\xi}$ sufficiently large and consequently $\sigma_{A_{-k_0}}(x,\xi)=0$ for all $\xi$ because it is a polynomial. Thus the order $N$ component of $A_{-k_0}$ vanishes at $x$, which implies that $A_{-k_0}|_W=0$. Since $x\in M$ was arbitrary, we conclude that $A_{-k_0}=0$ and hence $a_{-k_0}=0$, a contradiction.
\end{proof}

We can now prove our main result:

\begin{thm}\label{main theorem} Outside the set $-\frac{i}{2}\norm{\alpha_0}\N_0$, the scattering poles which are not in the physical plane $\mathcal{H}^\Gamma_\mathrm{phys}$ are precisely the quantum resonances. 

Moreover, if $\zeta_0\not\in -\frac{i}{2}\norm{\alpha_0}\N_0$ is a simple pole of the quantum resolvent as well as the scattering matrix, then the boundary value map induces an isomorphism between the images of the two residue operators at $\zeta_0$. 
\end{thm}

\begin{proof}
Let $\zeta_0$ be a quantum resonance of $\Gamma\backslash X$, i.e. a pole of $\tilde{r}_\zeta^\Gamma$. Then $\zeta_0$ is \emph{not} in the physical halfplane $\mathcal{H}^\Gamma_\mathrm{phys}$ and Lemma~\ref{lem:poles-diagonal} implies that $\zeta_0$ is also a pole of $\tilde{r}_\zeta^\Gamma\vert_{X\times X\backslash D_\Gamma}$. Now suppose that $\zeta_0\not\in -\frac{i}{2}\norm{\alpha_0}\N_0$.

\begin{description}
\item[Case 1  $c\big( i\zeta_0\frac{\alpha_0}{\norm{\alpha_0}}\big)\not=0$] In this case Equation~\eqref{eq:boundarymaps} and the result \cite[Thm.~6.4]{HHP19} imply that $\zeta\mapsto\beta_{\nu(\zeta)}$ is a holomorphic family of injective operators in a neighborhood of $\zeta_0$. Thus by Equation~\eqref{eq:result111}  $\zeta_0$ is a pole of 
$\tilde s_{\zeta}^\Gamma|_{\Omega\times \Omega\setminus D^\Omega_\Gamma}$. 
Hence  $\zeta_0$ is also a pole of $\tilde s_{\zeta}^\Gamma$, i.e. a scattering pole of $\Gamma\backslash X$. 

\item[Case 2  $c\big( i\zeta_0\frac{\alpha_0}{\norm{\alpha_0}}\big)=0$] 

In view of \cite[Re.~5.1]{HHP19} $c\big( i\zeta_0\frac{\alpha_0}{\norm{\alpha_0}}\big)=0$  implies $\Im\zeta_0>\|\rho\|$. If $\delta_\Gamma<2\|\rho\|$, then $i\rho\in \mathcal{H}^\Gamma_\mathrm{phys}$, hence also $\zeta_0\in \mathcal{H}^\Gamma_\mathrm{phys}$. Thus,  this case cannot occur, i.e. $\delta_\Gamma=2\|\rho\|$. But then $0$ is in the $L^2$-spectrum of $\Delta_{\Gamma\backslash X}$. Since the continuous $L^2$-spectrum of $\Delta_{\Gamma\backslash X}$ is $[\|\rho\|^2,\infty[$ the point $0$ has to be an eigenvalue. But  $\Delta_{\Gamma\backslash X}f=0$ for $f\in L^2(\Gamma\backslash X)$ is equivalent to $df=0$ and $d^*f=0$, so $f$ has to be constant. But this implies that $\Gamma\backslash X$ has finite volume, which contradicts our hypothesis that $\Gamma$ is convex-cocompact, but not cocompact. 
\end{description}

Conversely, assume that $\zeta_0$ is a scattering pole which is not in the physical halfplane $\mathcal{H}^\Gamma_\mathrm{phys}$ and satisfies  $\zeta_0\not\in -\frac{i}{2}\norm{\alpha_0}\N_0$. Then, by the argument given in Case 2, $c\big( i\zeta_0\frac{\alpha_0}{\norm{\alpha_0}}\big)\not=0$ and again Equation~\eqref{eq:boundarymaps} and \cite[Thm.~6.4]{HHP19} imply that $\zeta\mapsto\beta_{\nu(\zeta)}$ is a holomorphic family of injective operators in a neighborhood of $\zeta_0$. 

For any  $\lambda\in \aL_\C^\ast$ and $f\in  C^{-\infty}(\Gamma\backslash \Omega, V_{\Gamma\backslash \Omega}(\sigma_\lambda))$ \cite[p.~134]{BO12} yields an Eisenstein series, which we denote by $\mathcal P_\lambda^\Gamma f$ since it takes the role of the Poisson transform in the case of locally symmetric spaces. Then $\mathcal P_\lambda^\Gamma f\in  C^{\infty}(\Gamma\backslash X)$ is a $\Delta_{\Gamma\backslash X}$-eigenfunction, which has $f$ as boundary value. Applying this to the product $\Gamma\backslash \Omega\times \Gamma\backslash \Omega$ and scattering kernel $\tilde s_\zeta^\Gamma$ with the appropriate $\lambda$ determined by $\zeta$ we find 
a kernel function $r^{\Gamma,\mathrm{Eis}}_\zeta$ on $\Gamma\backslash X\times \Gamma\backslash X$ which has the same boundary values as $r^{\Gamma}_\zeta$.  
But then, since $\zeta\mapsto\beta_{\nu(\zeta)}$ in a neighborhood of $\zeta_0$ is a holomorphic family of injective maps, we have $r^{\Gamma,\mathrm{Eis}}_\zeta=r^{\Gamma}_\zeta$ and $\zeta_0$ is a pole of $r^{\Gamma}_\zeta$, i.e. it is also a quantum resonance. 

To prove the claim about the resonance multiplicities, let $\zeta_0\not\in -\frac{i}{2}\norm{\alpha_0}\N_0$ be a quantum resonance. Then we are in Case 1 above, so that $\zeta\mapsto\beta_{\nu(\zeta)}$ is a holomorphic family of injective operators in a neighborhood of $\zeta_0$. If $\zeta_0$ is a simple pole of  $\tilde r^\Gamma_{\zeta}|_{X\times X\setminus D_\Gamma}$, then the residue of  
$\beta_{\nu(\zeta)}(\tilde r^\Gamma_{\zeta}|_{X\times X\setminus D_\Gamma})$ at $\zeta_0$ is given by 
\[
\mathrm{Res}_{\zeta=\zeta_0}\big(\beta_{\nu(\zeta)}(\tilde r^\Gamma_{\zeta}|_{X\times X\setminus D_\Gamma})\big)
=
\beta_{\nu(\zeta_0)}\mathrm{Res}_{\zeta=\zeta_0}(\tilde r^\Gamma_{\zeta}|_{X\times X\setminus D_\Gamma}).
\]
On the other hand, we have 
\[
\mathrm{Res}_{\zeta=\zeta_0}(
a_\zeta \tilde s_{\zeta}^\Gamma|_{\Omega\times \Omega\setminus D^\Omega_\Gamma})
=  \frac{i}{2 \norm{\alpha_0}} \mathrm{Res}_{\zeta=\zeta_0}\big(
\frac{1}{\zeta} \tilde s_{\zeta}^\Gamma|_{\Omega\times \Omega\setminus D^\Omega_\Gamma}\big).\]
If $\zeta_0$ is also a simple pole of $ \tilde s_{\zeta}^\Gamma|_{\Omega\times \Omega\setminus D^\Omega_\Gamma}$, then
\[\frac{i}{2 \norm{\alpha_0}} \mathrm{Res}_{\zeta=\zeta_0}\big(
\frac{1}{\zeta} \tilde s_{\zeta}^\Gamma|_{\Omega\times \Omega\setminus D^\Omega_\Gamma}\big)
=
\frac{i}{2 \norm{\alpha_0}\zeta_0}\mathrm{Res}_{\zeta=\zeta_0}(
\tilde s_{\zeta}^\Gamma|_{\Omega\times \Omega\setminus D^\Omega_\Gamma}).
\]
Together, this gives 
\begin{equation}\label{eq:residues}
\beta_{\nu(\zeta_0)}\mathrm{Res}_{\zeta=\zeta_0}(\tilde r^\Gamma_{\zeta}|_{X\times X\setminus D_\Gamma})
=
\frac{i}{2 \norm{\alpha_0}\zeta_0}\mathrm{Res}_{\zeta=\zeta_0}(
\tilde s_{\zeta}^\Gamma|_{\Omega\times \Omega\setminus D^\Omega_\Gamma}).
\end{equation} 
We want to use this relation between kernel residues to get a relation between the images of the operator residues. To this end, we first observe that for all $\varphi,\psi\in \CT(X)$
\begin{align*}
\langle\mathrm{Res}_{\zeta=\zeta_0}(\tilde r^\Gamma_{\zeta}),\varphi\otimes\psi\rangle
&= \mathrm{Res}_{\zeta=\zeta_0}(\langle\tilde r^\Gamma_{\zeta},\varphi\otimes\psi\rangle)\\
&= \mathrm{Res}_{\zeta=\zeta_0}(\langle\tilde R^\Gamma_{\zeta}\varphi,\psi\rangle)\\
&= \langle\mathrm{Res}_{\zeta=\zeta_0}(\tilde R^\Gamma_{\zeta})\varphi,\psi\rangle.
\end{align*}
Thus, $\mathrm{Res}_{\zeta=\zeta_0}(\tilde r^\Gamma_{\zeta})$ is the Schwartz kernel of $\mathrm{Res}_{\zeta=\zeta_0}(\tilde R^\Gamma_{\zeta})$. Similarly, $\mathrm{Res}_{\zeta=\zeta_0}(\tilde s^\Gamma_{\zeta})$ is the Schwartz kernel of $\mathrm{Res}_{\zeta=\zeta_0}(\tilde{\mathbf{S}}^\Gamma_{\zeta})$,  where $\tilde {\mathbf{S}}^\Gamma_{\zeta}: C^\infty_c(\partial X)\to \mathcal D'(\partial X)$ is the linear operator with Schwartz kernel $\tilde s^\Gamma_{\zeta}$. 
To continue, let us use the simplified notation $\tilde{\mathsf R}:=\mathrm{Res}_{\zeta=\zeta_0}(\tilde R^\Gamma_{\zeta})$, $\tilde{\mathsf S}:=\mathrm{Res}_{\zeta=\zeta_0}(\tilde{\mathbf{S}}^\Gamma_{\zeta})$, $\tilde{\mathsf{r}}:=\mathrm{Res}_{\zeta=\zeta_0}(\tilde r^\Gamma_{\zeta})$, $\tilde {\mathsf{s}}:=\mathrm{Res}_{\zeta=\zeta_0}(\tilde s^\Gamma_{\zeta})$ as well as the somewhat abstract notation $\mathrm{dom}(\mathsf{X})$ for the domain of an operator $\mathsf{X}$. Then we have
\bq
\mathrm{im}(\tilde{\mathsf R})=\mathrm{span}\{\tilde{\mathsf r}(f\otimes \cdot)\,|\, f\in \mathrm{dom}(\tilde{\mathsf R})\},\label{eq:5210529510}
\eq
while
\begin{align*}
\mathrm{im}(\tilde{\mathsf S})&=\mathrm{span}\{\tilde{\mathsf s}(\omega\otimes \cdot)\,|\, \omega\in \mathrm{dom}(\tilde{\mathsf S})\}\\
&\stackrel{\hspace*{-1em}\text{\eqref{eq:residues}}\hspace*{-1em}}{=}\mathrm{span}\{\beta_{\nu(\zeta_0)}(\tilde{\mathsf r})(\omega\otimes \cdot)\,|\, \omega\in \mathrm{dom}(\tilde{\mathsf S})\}\\
&\stackrel{\hspace*{-1em}\text{\eqref{eq:boundarymaps}}\hspace*{-1em}}{=}\mathrm{span}\{\beta^{\mathrm{HHP}}_{\rho-i\frac{\zeta_0 \alpha_0}{\norm{\alpha_0}}}\big(\tilde{\mathsf r}(\beta^{\mathrm{HHP}'}_{\rho-i\frac{\zeta_0 \alpha_0}{\norm{\alpha_0}}}(\omega)\otimes \cdot)\big)\,|\, \omega\in \mathrm{dom}(\tilde{\mathsf S})\}\\
&=\mathrm{span}\{\beta^{\mathrm{HHP}}_{\rho-i\frac{\zeta_0 \alpha_0}{\norm{\alpha_0}}}\big(\tilde{\mathsf r}(f\otimes \cdot)\big)\,|\, f\in \mathrm{dom}(\tilde{\mathsf R})\}\\
&\stackrel{\hspace*{-1em}\text{\eqref{eq:5210529510}}\hspace*{-1em}}{=}\beta^{\mathrm{HHP}}_{\rho-i\frac{\zeta_0 \alpha_0}{\norm{\alpha_0}}}\big(\mathrm{im}(\tilde{\mathsf R})\big),
\end{align*}
where in the second last step we used that the adjoint $\beta^{\mathrm{HHP'}}_{\rho-i\frac{\zeta_0 \alpha_0}{\norm{\alpha_0}}}:\mathrm{dom}(\tilde{\mathsf S})\to\mathrm{dom}(\tilde{\mathsf R})$ is surjective (being the adjoint of an injective map). Now, analogously as at the end of the proof of \cite[Thm.~7.1]{HHP19}, we observe that as a map $\mathrm{im}(\tilde{\mathsf R})\to\mathrm{im}(\tilde{\mathsf S})$ the boundary value map is inverted by the Poisson transform not only from the right but also from the left, so that it is bijective. Thus, passing to $\Gamma$-quotients using the $\Gamma$-equivariance of the boundary value map, we arrive at a linear isomorphism
\[
\mathrm{im}(\mathrm{Res}_{\zeta=\zeta_0}(\mathbf{S}_{\zeta}^\Gamma))\cong\mathrm{im}(\mathrm{Res}_{\zeta=\zeta_0}(R_{\zeta}^\Gamma)).
\]
This finishes the proof.
\end{proof}

\providecommand{\bysame}{\leavevmode\hbox to3em{\hrulefill}\thinspace}
\providecommand{\MR}{\relax\ifhmode\unskip\space\fi MR }
\providecommand{\MRhref}[2]{%
  \href{http://www.ams.org/mathscinet-getitem?mr=#1}{#2}
}
\providecommand{\href}[2]{#2}

\bigskip

\end{document}